\documentclass[11pt, a4paper]{article}
\usepackage{amsmath}
\usepackage{amsfonts}
\usepackage{amsthm}
\usepackage{authblk}
\usepackage{here}
\usepackage[pdftex]{graphicx}

\newtheorem{prop}{Proposition}[section]
\newtheorem{lem}[prop]{Lemma}
\newtheorem{thm}[prop]{Theorem}
\newtheorem{cor}[prop]{Corollary}
\newtheorem{theorem}{Theorem}

\theoremstyle{definition}
\newtheorem{definition}[prop]{Definition}
\newtheorem{exam}[prop]{Example}

\begin{document}

\title{GEOMETRIC APPROACH TO\\ GRAPH MAGNITUDE HOMOLOGY}
\author[$\star$]{Yasuhiko Asao}
\author[$\dagger$]{Kengo Izumihara} 
\affil[$\star$]{Graduate School of Mathematical Sciences, University of Tokyo, Tokyo, Japan \texttt{ (asao@ms.u-tokyo.ac.jp)}}
\affil[$\dagger$]{Graduate School of Mathematics, Kyushu University, Fukuoka, Japan \texttt{ (izumihara.kengo.587@s.kyushu-u.ac.jp)}}
\date{}
\maketitle

\begin{abstract}
In this paper, we introduce a new method to compute magnitude homology of general graphs. 
To each direct sum component of magnitude chain complexes, 
we assign a pair of simplicial complexes whose simplicial chain complex is isomorphic to it.
First we states our main theorem specialized to trees, which gives another proof for the known fact that trees are diagonal.  After that, we consider general graphs, which may have cycles. We also demonstrate some computation as an application.
\end{abstract}

\section{Introduction}
Leinster (\cite{Lein}) introduced magnitude of finite metric spaces which measures `` the number of efficient points". Magnitude homology has been invented as a categoryfication of magnitude of a graph which is equipped with a graph metric, by Hepworth-Willerton (\cite{Hep}). 
Magnitude homology $MH_{k,\ell}(G)$ of a graph $G$ is defined by the $k$-th homology group of a chain complex $MC_{\ast,\ell}(G)$, whose chain groups are generated by tuples of vertices of length $\ell$.

Several tools for computing magnitude homology of a graph have been studied so far.
For examples, Hepworth-Willerton (\cite{Hep}) proves a Mayer-Vietoris type exact sequence and a K\"{u}nneth type formula, and Gu (\cite{Gu}) uses algebraic Morse theory for computation for some graphs.
Although, in general, computation of magnitude homology remains a difficult problem.

In this paper, we introduce another method to compute magnitude homology of general graphs. 
Our strategy is to replace the computation of magnitude chain complex $MH_{k,\ell}(G)$ by that of simpicial homology. 
A similar method using order complex is studied by Kaneta-Yoshinaga (\cite{Yos}), 
whereas we assign simplicial complexes in another way.
A subtle difference from Kaneta-Yoshinaga's method which restricts us to work within a range with no 4-cuts, is that our method can be applied to general graphs.

For a magnitude chain complex $MC_{\ast, \ell}(G)$, we denote by $MC_{\ast,\ell}(a, b)$ a direct sum component of it, which consists of tuples with ends $a$ and $b$. Our main result is the following which appears as Theorem \ref{thm4.2} in this paper. We assume that graphs are connected and contain no loops.
\begin{theorem}
\label{thm1.1}
Let $a, b$ be vertices of a graph $G$, and fix an integer  $\ell\ge 3$. 
Then we can construct a pair of simplicial complexes $(K_\ell(a, b),K'_\ell(a, b))$ which satisfies
\[
C_\ast(K_\ell(a, b), K'_{\ell}(a, b)) \cong MC_{\ast+2,\ell}(a, b).
\]
In particular, we have
\[
MH_{k, \ell}(a, b)\cong H_{k-2}(K_\ell(a, b), K'_\ell(a, b))
\]
for $k,\ell\ge3$. Moreover, for $k = 2$, we also have
\[
MH_{2, \ell}(a, b) \cong \begin{cases}H_{0}(K_\ell(a, b), K'_\ell(a, b)) & {\rm if}\ d(a, b) < \ell, \\ \tilde{H}_{0}(K_\ell(a, b))& {\rm if}\ d(a, b) = \ell,\end{cases}
\]
where $\tilde{H}_{\ast}$ denotes the reduced homology group.
\end{theorem}

Our theorem yields an interpretation of magnitude homology groups as a homology group of a simplicial complex.
Therefore, our method allows us to apply sophisticated tools of homotopy theory.
In the special cases of $2 \leq \ell \leq 4$ we obtain a visualization of the magnitude chain complex
since the dimensions of the corresponding simplical complexes are 0, 1,  and 2, respectively.

The organization of this paper is the following:  After giving some basic definitions and notations in section 2, we first give a new method for computing magnitude homology of trees based on our simplicial strategy. Obtained computational results coincide with \cite[Corollary 31]{Hep}. The computation for a tree is simpler than that for general graphs studied in the following section,  because of the fact that the magnitude chain complex of a tree can be decomposed into simple ones. In section 4, we give a proof of our main theorem, and compute the magnitude homology of the graph $Sq_{2}$ introduced in \cite{Lein2} as an application.

\section{Preliminaries}
In this section, we recall some basic definitions for graphs and their magnitude homology together with related notation. Main definitions are taken from \cite{Hep}.
\subsection{Simplicial complexes}
\begin{definition}
Let $V$ be a set, and let $P(V)$ be its power set. A subset $S \subset P(V)\setminus \{\emptyset\}$ is called a \textit{simplicial \ complex} if it satisfies that
\[
B \in S {\rm \ for \ every\ }\emptyset \neq B \subset A \in S .
\]
A subset $S'$ of a simplicial complex $S$ is called {\it a subcomplex of } $S$ if $S'$ itself is a simplicial complex.
\end{definition}
\begin{definition}
For a simplicial complex $S\subset P(X)\setminus \{\emptyset\}$, we associate a chain complex $(C_{\ast}(S), \partial_{\ast})$ defined as follows:
\[
\begin{cases}
C_{n}(S) = \mathbb{Z}\langle \sigma \in S \mid  \# \sigma = n+1\rangle, \\
\partial_{n}\{s_{0}, \dots s_{n}\} := \sum_{i = 0}^{n} (-1)^{i}\{s_{0}, \dots, \hat{s}_{i}, \dots, s_{n}\},
\end{cases}
\]
where the index $s_{i}$ is a fixed total order on $X$, the notation $\hat{s}_{i}$ means the removal of the vertex $s_{i}$, and $\#$ denotes the cardinality of a set. For a subcomlex $S'$ of $S$, the associated chain complex $C_{\ast}(S')$ is obviously a subcomplex of $C_{\ast}(S)$. We define
\[
C_{\ast}(S, S') := C_{\ast}(S)/C_{\ast}(S').
\]
\end{definition}
We suppose that any chain complex $C_{\ast}$ has no negative component, that is, $C_{i} = 0$ for $i < 0$. For a chain complex $(C_{\ast}, \partial_{\ast})$, we denote by $C_{\ast+N}$ a chain complex $(D_{\ast}, \partial'_{\ast})$ defined as follows:
\begin{align*}
D_{i} = \begin{cases} C_{i+N} & (i\geq 0), \\ 0 & (i < 0), \end{cases} \\
\partial'_{i} = \begin{cases} \partial_{i+N} & (i\geq 0), \\ 0 & (i < 0). \end{cases}
\end{align*}
\subsection{Graphs}
\begin{definition}
For a simplicial complex $S$, an element $A \in S$ with $\# A = 1$ is called \textit{a vertex} of $S$, and an element $A \in S$ with $\# A = 2$ is called \textit{an edge} of $S$.
\end{definition}
\begin{definition}
\label{def2.0}
A \textit{graph} $G$ is a simplicial complex with $\# A \leq 2$ for every $A \in S$. We denote by $V(G)$ the set of vertices of $G$, and denote by $E(G)$ the set of edges of $G$, which are called \textit{a vertex set} and \textit{an edge set} respectively.
\end{definition}
\begin{definition}
\label{def2.03}
\begin{enumerate}
\item A graph $G$ is \textit{finite} if its vertex set $V(G)$ is a finite set.
\item A graph $G$ is \textit{connected} if any two vertices $a, b\in V(G)$ are connected, that is, there exists a finite sequence of edges 
\[
e_1, \dots, e_n \in E(G),
\] 
with 
\[
\begin{cases}a \in e_1, b \in e_n , \\ e_i \cap e_j \ne \emptyset.\end{cases}
\]
\item A \textit{cycle} in a graph $G$ is a finite sequence of edges
\[
e_{1}, \dots, e_{n}\in E(G),
\]
with 
\[
\begin{cases} e_1 \cap e_n \ne \emptyset, \\ e_i \ne e_j & (i \ne j).\end{cases}
\]
\item A graph is a \textit{tree} if it is finite, connected, and contains no cycles.
\end{enumerate}
\end{definition}
Throughout the paper, we assume that graphs are connected.
\begin{definition}
\label{def2.06}
For a graph $G$, we define a distance function $d : V(G)\times V(G)\longrightarrow\mathbb{Z}_{\ge 0}$ as follows:
\begin{itemize}
\item In case $a\ne b$, we define $d(a, b)=n$ where $n$ is the smallest integer such that there exist a sequence of edges
\[
e_1, \dots, e_n \in E(G),
\] 
with
\[
\begin{cases}a \in e_1, b \in e_n , \\ e_i \cap e_j \ne \emptyset.\end{cases}
\]
\item In case $a = b$, we define $d(a, b)=0$.
\end{itemize}
\end{definition}
\subsection{Magnitude homology of graphs}
\begin{definition}
For a graph $G$, we call a tuple $(x_{0}, \dots, x_{k})\in V(G)^{k+1}$ a $(k+1)$-{\it sequence} if it satisfies $x_j\ne x_{j+1}$ for every $0\leq  j \leq  k-1$.
\end{definition}
\begin{definition}
\label{def2.1}
Let $G$ be a graph and fix an integer $\ell\ge 0$. \textit{Magnitude chain complex of length} $\ell$ of $G$ is defined as follows. We denote it by $MC_{\ast,\ell}(G)$. 
The graded module $MC_{\ast,\ell}(G)$ is defined as the family of free $\mathbb{Z}$-modules $\{MC_{k, \ell}(G)\}_{k\ge 0}$ generated by all $(k+1)$-sequences
\[ \mathbf{x}=(x_0, \dots, x_k) \in V(G)^{k+1} \]
satisfying
\[
\sum_{i=0}^{k-1}d(x_i, x_{i+1}) = \ell.
\]
The boundary map is defined by $\partial = \textstyle\sum_{i=1}^{k-1} (-1)^i \partial_i$ with
\[
\partial_i(\mathbf{x}) := \begin{cases}
(x_0, \dots, \hat{x}_{i}, \dots, x_k) & \textrm{if }d(x_{i-1}, x_{i+1}) = d(x_{i-1}, x_i) + d(x_i, x_{i+1}),\\
0 & \textrm{otherwise,}
\end{cases}
\]
where the notation $\hat{x}_{i}$ means the removal of the vertex $x_i$.
\end{definition}
It has been proved that $\partial^2 = 0$ in \cite[Lemma 11]{Hep}.
\textit{Magnitude homology} $MH_{k, \ell}(G)$ of a graph $G$ is defined as the homology group $H_k(MC_{\ast, \ell})$.
For convenience, we define the \textit{length function} $\mathrm{L}$ by
\[ \mathrm{L}(\mathbf{x})
:=\textstyle\sum_{i=0}^{k-1}d(x_i, x_{i+1})
\]
for $\mathbf{x}\in V(G)^{k+1}$, and we call it the {\it length} of $\mathbf{x}$.
The condition $d(x_{i-1}, x_{i+1}) = d(x_{i-1}, x_i) + d(x_i, x_{i+1})$ in Definition \ref{def2.1} is equivalent to the condition
\[
\mathrm{L}(x_0, \dots ,\hat{x_i}, \dots, x_k)=\mathrm{L}(x_0, \dots, x_k).
\]
 
By definition, we have the following proposition.
\begin{prop}
\label{prop2.2}
For $\ell\ge 0$, we have the direct sum decomposition of a magnitude chain complex
\[ MC_{\ast, \ell}(G)
=\bigoplus_{a, b\in V(G)}
MC_{\ast, \ell}(a, b),
\]
where $MC_{\ast, \ell}(a, b)$ is the subcomplex of $MC_{\ast, \ell}(G)$ generated by sequences which start at $a$ and end at $b$.
\end{prop}

Hence the computation of magnitude homology of a graph $G$ reduces to the computation of each $(a, b)$-component. We define a subsequence of a sequence as follows.

\begin{definition}\label{def2.3}

Let $\mathbf{x}=(x_0,  \dots , x_k)$ be a sequence, and $\mathbf{y}=(y_0, \dots , y_{k'})$ be a tuple. We call the tuple $\mathbf{y}$ a \textit{subsequence} of $\mathbf{x}$ if there exists integers $0 = i_0 < \dots < i_{k'} = k$ such that $x_{i_j} = y_j$ for each $0 \leq j \leq k'$.
When $\mathbf{y}$ is a subsequence of $\mathbf{x}$, we denote it by $\mathbf{y}\prec\mathbf{x}$.

\end{definition}
Note that a subsequence need not to be a sequence.
\begin{definition}\label{ind}

For a subsequence $\mathbf{y} = (x_0, x_{j_1}, \dots , x_{j_{k'}}, x_k)\prec\mathbf{x}$, we call a  set $\{j_1, \dots, j_{k'}\}$ \textit{the indices} of $\mathbf{y}\prec\mathbf{x}$.

\end{definition}

\section{Computation for trees}

In this section, we compute magnitude homology of a tree which is known in \cite[Corollary 31]{Hep} by using simplicial homology. 
\begin{definition}\label{def3.1}

Let $k\ge1$. We call a sequence $(x_0, x_1, \dots , x_k) \in V(G)^{k+1}$  \textit{a path} in a graph $G$ if it satisfies
\[ 
d(x_i, x_{i+1}) = 1,
\]
for every $0 \leq i \leq k-1$. For vertices $a, b\in V(G)$, we denote by $P_{\le \ell}(a, b)$ the set of  all paths $(x_0, \dots, x_k)$ in $G$ satisfying
\[
\begin{cases}
 x_0 = a, x_k=b \\
 \mathrm{L}(x_0, \dots, x_k)=k\le \ell.
 \end{cases}
 \]
\end{definition}

Note that, for each sequence $\mathbf{x} = (x_0,  \dots, x_k)$, there exists a shortest path which passes through vertices $x_0, \dots, x_k$ in this order.  We call such a shortest path \textit{a path of} $\mathbf{x}$. If $G$ is a tree, a path of $\mathbf{x}$ is unique for each sequence $\mathbf{x}$. Let $G$ be a tree. For a path $\mathbf{x} = (a, x_1, \dots, x_{k-1}, b) \in V(G)^{k+1}$, we denote by $MC_{k, \ell}(\mathbf{x})$ the submodule of $MC_{k, \ell}(a, b)$ generated by sequences whose paths coincide with $\mathbf{x}$. Clearly, $MC_{\ast,\ell}(\mathbf{x})$ is a subcomplex of $MC_{\ast, \ell}(a, b)$ since each $\partial_i(\mathbf{x})$ is 0 or  has $\mathbf{x}$ as its path.

\begin{prop}\label{prop3.2}

Let $G$ be a tree. We have the following direct sum decomposition
\[ MC_{\ast, \ell}(a, b)
=\bigoplus_{\mathbf{x} \in P_{\le \ell}(a, b)}
MC_{\ast, \ell}(\mathbf{x}), 
\]
for each $a, b\in V(G)$ and $\ell \ge 1$.
\end{prop}

\begin{proof}

Since we have seen that each sequence belongs to the unique component of the decomposition, it is sufficient to see that $\partial\mathbf{y}\in MC_{k-1, \ell}(\mathbf{x})$ for $\mathbf{y}\in MC_{k, \ell}(\mathbf{x})$.
Let $\mathbf{y} =(y_0, \dots, y_k)\in MC_{k, \ell}(\mathbf{x})$.
Then we have 
\[ \partial\mathbf{y}
=\sum_{\substack{\mathrm{L}(y_0, \dots,\hat{y_{i}},  \dots, y_k)
=\ell\\1 \le i \le k-1}}(-1)^i(y_0,  \dots,\hat{y_{i}},  \dots, y_k).
 \]
Since the path of each sequence $(y_0, \dots,\hat{y_{i}},  \dots, y_k)$ is unique, it must coincide with $\mathbf{x}$ if the length $\mathrm{L}(y_0,  \dots,\hat{y_{i}},  \dots, y_k)$ is preserved. Therefore we obtain that $\partial\mathbf{y}\in MC_{k-1, \ell}(\mathbf{x})$.
\end{proof}

In the following, we will construct a pair of simplicial complexes whose associated chain complex is isomorphic to  the magnitude chain complex $MC_{\ast, \ell}(\mathbf{x})$ for each $\ell\ge 3$ and for each path $\mathbf{x}$ in $G$. For a path $\mathbf{x} = (x_{0}, \dots x_{\ell})$, we consider a subsequence
$\varphi(\mathbf{x})=(x_0, x_{i_1}, \dots, x_{i_m}, x_\ell) \prec \mathbf{x}$ satisfying

\[ 
\begin{cases}
0<i_s<\ell,  \\
d(x_{i_s-1}, x_{i_s+1}) < d(x_{i_s-1}, x_{i_s}) + d(x_{i_s}, x_{i_s+1}),
\end{cases}
\]
for every $1 \leq s \leq m$. If $G$ is a tree, it turns out that $\varphi(\mathbf{x})$ consists of all  ``turning points" of $\mathbf{x}$ and end points by the following lemma.
\begin{lem}
\label{lem3.3}
Let $G$ be a tree,  and let $\mathbf{x} = (x_{0}, \dots x_{\ell})$ be a path in $G$. For every $1\leq i \leq \ell -1$, we have $(x_{0}, x_i, x_{\ell})\prec\varphi(\mathbf{x})$ if and only if $x_{i-1}=x_{i+1}$.
\end{lem}
\begin{proof}
Every consecutive three points of a path in a tree must have either of the configuration of Figure 1.
\begin{center}
\begin{figure}[H]
\includegraphics[width=12cm]{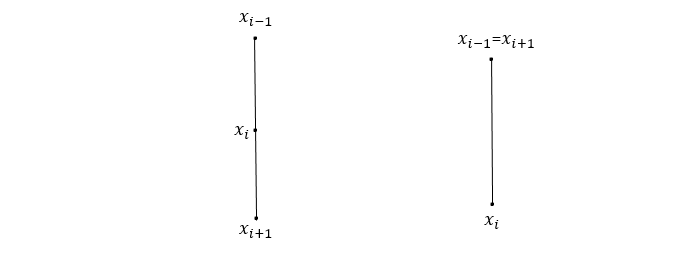}
\caption{Consecutive three points in a tree.}
\end{figure}
\end{center}
Thus we have $x_{i-1}=x_{i+1}$ if and only if the triangle inequality is not an equality.
\end{proof}

Let $\Delta^{\ell-2}$ be the standard $(\ell - 2)$- simplex $P(\{ 1, \dots,\ell-1 \}) \setminus {\emptyset}$. For a path $\mathbf{x} = (x_{0}, \dots x_{\ell})$ and its subsequence $\varphi(\mathbf{x})=(x_0, x_{i_1}, \dots, x_{i_m}, x_\ell)$, we define a subset $\Delta_{\mathbf{x}} \subset \Delta^{\ell - 2}$ by
\[ 
\Delta_{\mathbf{x}} = \{ \sigma \mid \{i_{1}, \dots, i_{m} \} \not \subset \sigma \in \Delta^{\ell - 2} \}.
 \]
For every $\sigma \in \Delta_{\mathbf{x}}$, any subset $\sigma' \subset \sigma$ is a simplex of  $\Delta_{\mathbf{x}}$, which implies that $\Delta_{\mathbf{x}}$ is a subcomplex of $\Delta^{\ell - 2}$.

\begin{prop}
\label{prop3.4}
There exists a chain isomorphism 
\[ 
(MC_{\ast + 2,\ell}(\mathbf{x}),\partial)\cong (C_{\ast}(\Delta^{\ell-2}, \Delta_{\mathbf{x}}),-\partial) .
\]
 
\end{prop}
\begin{proof}
Note that every sequence $\mathbf{y}$ belonging to $MC_{\ast + 2,\ell}(\mathbf{x})$ is a subsequence of $\mathbf{x}$, that is
\[
\mathbf{y} = (x_0,x_{j_1}, \dots, x_{j_{k}}, x_\ell) \prec \mathbf{x}.
\]
We define a homomorphism
\[ 
t : MC_{\ast + 2,\ell}(\mathbf{x})\longrightarrow C_{\ast}(\Delta^{\ell-2}, \Delta_{\mathbf{x}})
 \]
by sending each sequence to its indices (Definition \ref{ind})
\[ 
(x_0,x_{j_1}, \dots, x_{j_{k}}, x_\ell)\longmapsto \{j_1, \dots, j_{k}\}
 \]
and extending it linearly. We can easily see that this is well-defined by the definitions. We show that homomorphism $t$ is bijective and is a chain map. First we show the injectivity.
Suppose that we have
\[
t\left(\sum_{\alpha=1}^N c_\alpha(x_0,x_{j_1^\alpha}, \dots,x_{j_{k}^\alpha}, x_{\ell})\right)=0.
\]
In the case that 
\[
\{j_1^\alpha, \dots, j_{k}^\alpha\} = \{j_1^{\alpha'}, \dots, j_{k}^{\alpha'}\}
\]
for every $1\leq \alpha, \alpha' \leq N$, it is clear that 
\[
\sum_{\alpha=1}^N c_\alpha(x_0, x_{j_1^\alpha}, \dots, x_{j_{k}^\alpha}, x_{\ell})=0.
\]
 In general, the index set $\{1, \dots, N\}$ can be decomposed into pairwise disjoint subsets $A_{0}, \dots, A_{M}$ such that $\alpha, \alpha' \in A_{m}$ implies that 
 \[
 \{j_1^\alpha, \dots, j_{k}^\alpha\} = \{j_1^{\alpha'}, \dots, j_{k}^{\alpha'}\},
 \]
 for every $0 \leq m \leq M$. Hence we obtain
\begin{align*}
\sum_{\alpha=1}^N c_\alpha(x_0, x_{j_1^\alpha}, \dots, x_{j_{k}^\alpha}, x_{\ell}) = \sum_{m = 0}^{M}\sum_{\alpha \in A_{m}} c_\alpha(x_0, x_{j_1^\alpha}, \dots, x_{j_{k}^\alpha}, x_{\ell}) = 0,
\end{align*}
which implies that $t$ is injective. Next we show the surjectivity. By definition,  $[\{j_1, \dots, j_{k}\}] \ne 0 \in  C_{\ast }(\Delta^{\ell-2}, \Delta_{\mathbf{x}})$ implies that
\[
\varphi(\mathbf{x}) \prec (x_0, x_{j_1}, \dots, x_{j_{k}}, x_\ell).
\]
By the definition of $\Delta_{\mathbf{x}}$, the paths of $\varphi(\mathbf{x})$ and $(x_0, x_{j_1}, \dots, x_{j_{k}}, x_\ell)$ coincide, which implies that 
\[
(x_0,x_{j_1},\dots, x_{j_{k}}, x_\ell) \in MC_{k,\ell}(\mathbf{x}).
\]
Hence we see that $t$ is surjective. Finally, the following calculation shows that $t$ is a chain map :
\begin{align*}
&\partial t\left(\sum_{\alpha}c^\alpha(x_0, x_{j_1}^{\alpha}, \dots, x_{j_{k}}^{\alpha}, x_\ell)\right)\\
&= \partial \left(\sum_{\alpha}c^\alpha \{j_{1}^\alpha,  \dots, j_{k}^\alpha\}\right)\\
&= \sum_{\alpha}c^\alpha \sum_{\substack{0 < i \leq k \\ \{i_1, \dots, i_m\}\subset\{j_1^\alpha,   \dots,\hat{j_{i}}^{\alpha}, \dots, j_{k}^\alpha\}}}
(-1)^{i-1}\{j_1^\alpha, \dots,\hat{j_{i}}^{\alpha}, \dots, j_{k}^\alpha\}\\
&= \sum_{\alpha}c^\alpha \sum_{\substack{0 < i \leq k \\ 
d(x^{\alpha}_{j_{i}-1}, x^{\alpha}_{j_{i}}) + d(x_{j_{i}}^{\alpha}, x^{\alpha}_{j_{i}+1})=\\
d(x^{\alpha}_{j_{i}-1}, x^{\alpha}_{j_{i}+1})}}
(-1)^{i-1}\{j_{1}^\alpha, \dots, \hat{j_{i}}^{\alpha}, \dots, j_{k}^\alpha\}\\
& = t\left( -\partial \sum_{\alpha}c^\alpha(x_0, x_{j_1}^{\alpha}, \dots, x_{j_{k}}^\alpha  ,x_\ell )\right).
\end{align*}
\end{proof}
To compute the homology of $C_{\ast}(\Delta^{\ell-2}, \Delta_{\mathbf{x}})$, we use the homology exact sequence
\[ 
\cdots\longrightarrow H_k(\Delta^{\ell-2}) \longrightarrow H_k(\Delta^{\ell-2}, \Delta_{\mathbf{x}}) \longrightarrow H_{k-1}(\Delta_{\mathbf{x}}) \longrightarrow H_{k-1}(\Delta^{\ell-2})\longrightarrow\cdots.
 \]
Since $\Delta^{\ell-2}$ is contractible, we have $H_k(\Delta^{\ell-2},\Delta_{\mathbf{x}})\cong H_{k-1}(\Delta_{\mathbf{x}})$ while $k>1$.

Now we determine the homotopy type of $\Delta_{\mathbf{x}}$. When $m=0$, we have $\Delta_{\mathbf{x}}=\emptyset$. When $m=\ell-1$, we can see that $\Delta_{\mathbf{x}}$ is homotopy equivalent to the sphere $S^{(\ell-3)}$.
The following proposition shows that it is contractible in the other cases.
\begin{prop}\label{prop3.5}

For every path $\mathbf{x}$, the complex $\Delta_{\mathbf{x}}$ is contractible when $0 < m<\ell-1$ where $\varphi(\mathbf{x})=(x_0, x_{i_1}, \dots, x_{i_m}, x_\ell)$.
\end{prop}
\begin{proof}
The complex $\Delta_{\mathbf{x}}$ can be obtained from maximal faces $\{1, \dots,\hat{i}_{j}, \dots,\ell-1\}$ for $1 \leq j \leq m$ by glueing them at the common face $\{m+1,m+2, \dots,\ell-1\}$ (bold parts in case $m=2,3$ of Figure 2). Since each maximal face is a deformation retract of the  contractible common face, the whole complex is also contractible.
\begin{center}
\begin{figure}[h]
\includegraphics[width=12cm]{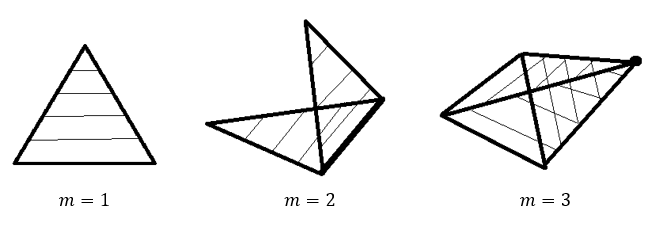}
\caption{Case $l=5$.}
\end{figure}
\end{center}
\end{proof}
Now we can completely compute magnitude homology of trees, which has been obtained by several authors.
\begin{thm}
\label{thm3.6}
Let $G$ be a tree, and let $k, \ell \geq 3$ be integers.
Then we have
\[
MH_{k,\ell}(G)=\begin{cases}
\mathbb{Z}^{2\# E(G)}, & k=\ell,\\
0, & k\ne\ell,
\end{cases}
 \]
where $\# E(G)$ denotes the cardinality of the edge set $E(G)$.
\end{thm}
\begin{proof}
By Proposition \ref{prop2.2} and \ref{prop3.2}, we have
\[ 
MH_{k,\ell}(G)=\bigoplus_{a,b\in V(G)}\bigoplus_{\mathbf{x}\in P_{\le\ell}(a, b)}MH_{k,\ell}(\mathbf{x}).
 \]
If $\mathbf{x}$ is a path such that $\varphi(\mathbf{x})$ has less than $(\ell+1)$ points,
then the $\mathbf{x}$-component $MH_{k,\ell}(\mathbf{x})$ is trivial by Proposition \ref{prop3.4} and \ref{prop3.5}.
In case $\varphi(\mathbf{x})$ has $(\ell+1)$ points,  we have that
\[ 
x_0=x_2=x_4= \dots\textrm{ and }x_1=x_3=x_5= \dots
 \]
by Lemma \ref{lem3.3}. We can easily see that the amount of such paths is two times the cardinality of the edge set $E(G)$.  Then by Proposition \ref{prop3.4}, we have
\begin{align*}
MH_{\ell,\ell}(G) &=  \bigoplus_{\{x_0, x_1\}\in E(G)} MH_{\ell,\ell}(x_0, x_1, x_0, \dots)\oplus MH_{\ell,\ell}(x_1, x_0, x_1, \dots) \\
&\cong \bigoplus_{\{x_0, x_1\}\in E(G)} H_{\ell-3}(S^{(\ell-3)}) \oplus H_{\ell-3}(S^{(\ell-3)})\\
&\cong  \mathbb{Z}^{2\# E(G)}.
\end{align*}
\end{proof}

\section{Theory for general graphs}
For general graphs, we cannot decompose magnitude homology indexed by paths as in the case of trees, since sequences may have more than one shortest paths. 
Hence we develop a method to compute each $(a, b)$-component in a similar way as in the tree case.
Let $G$ be a connected graph and  let $a, b\in V(G)$. We fix an integer $\ell \ge 3$.
\begin{definition}
\label{def4.0}
Let $K_\ell(a, b)$ be the set whose elements are subsets
\[
\{(x_{i_1},i_1), \dots, (x_{i_k},i_{k})\}\subset V(G)\times\{1,2, \dots,\ell-1\}
\]
such that there exists a path
\[
(a, x_1, \dots, x_{\ell'-1}, b)\in P_{\le \ell}(a, b)
\]
with  $(a, x_{i_1}, \dots, x_{i_{k}}, b)\prec (a, x_1, \dots, x_{\ell'-1}, b)$.
\end{definition}
For simplicity, we abbreviate  $\{(x_{i_1},i_1), \dots, (x_{i_k},i_k)\}$ to $\{x_{i_1}, \dots, x_{i_k}\}$ if there is no confusion. The set $K_\ell(a, b)$ is a simplicial complex since
\[
\{x_{j_1}, \dots, x_{j_{k'}}\} \subset \{x_{i_1}, \dots, x_{i_{k}}\}\in K_l(a, b)
\] 
implies that there exists  a path $(a, x_1, \dots, x_{\ell'-1}, b)\in P_{\le \ell}(a, b)$ with 
\[
(a, x_{j_1}, \dots, x_{j_{k'}}, b) \prec (a, x_{i_1}, \dots, x_{i_{k}}, b)\prec (a, x_1,  \dots, x_{\ell'-1}, b), 
 \]
that is  
\[
\{x_{j_1},  \dots, x_{j_{k'}}\} \in K_\ell(a, b).
\]
Clearly, the complex $K_{\ell-1}(a, b)$ is a subcomplex of $K_\ell(a, b)$.  We also define a subcomplex $K'_\ell(a, b)\subset K_\ell(a, b)$ by
\[
K'_\ell(a, b):=
\{\{x_{i_1}, \dots, x_{i_k}\}\in K_\ell(a, b)\mid \mathrm{L}(a, x_{i_1}, \dots, x_{i_k}, b)\le \ell-1\}.
\]
Our goal is to construct an isomorphism between $MC_{\ast,\ell}(a, b)$ and the quotient chain complex $C_{\ast - 2}(K_\ell(a, b), K'_{\ell}(a, b))$. We assume that $d(a, b)\le \ell$, since we have $K_\ell(a, b)=\emptyset$ for $d(a, b)>\ell$.
\begin{lem}\label{lem4.1}

Let $\{x_{i_1}, \dots, x_{i_k}\}$ and $\{y_{j_1}, \dots, y_{j_k}\}$ be simplices of $K_\ell(a, b)$. 
If we have
\[
\mathrm{L}(a, x_{i_1}, \dots, x_{i_k},b)=\mathrm{L}(a, y_{j_1}, \dots, y_{j_k},b)=\ell
\]
and $x_{i_s}=y_{j_s}$ for $1\leq s \leq   k$, then we have 
\[ 
i_s=j_s, 
\]
for $1\leq s \leq k$.
\end{lem}
\begin{proof}
By definition, there is a path
\[
(a, x_1, \dots, x_{\ell-1}, b)\in P_{\le \ell}(a.b) 
 \]
with 
\[
(a, x_{i_1}, \dots, x_{i_k}, b) \prec (a, x_1, \dots, x_{\ell-1}, b).
 \]
For $1 \leq s \leq k$, we have
\[
i_s = i_{s-1} + d(x_{i_{s-1}}, x_{i_s})=\sum_{n=0}^{s-1} d(x_{i_n}, x_{i_{n+1}}).
\]
Similarly we also have 
\[
j_s= \sum_{n=0}^{s-1}d(y_{j_n}, y_{j_{n+1}}),
\]
for $1 \leq s \leq k$.
Since we have $x_{i_s} = y_{j_s}$ for $1 \leq s \leq k$, we obtain $i_s=j_s$.
\end{proof}
Next we define a homomorphism
\[
t : (C_\ast(K_\ell(a, b), K'_{\ell}(a, b)), -\partial )\longrightarrow (MC_{\ast+2,\ell}(a, b),\partial )
 \]
by
\[
[\{x_{i_1}, \dots, x_{i_{k}}\}] \longmapsto (a, x_{i_1}, \dots, x_{i_{k}}, b).
 \]
It is well-defined since $\partial_i (a, x_{i_1}, \dots, x_{i_{k}}, b)$ vanishes exactly when $\partial_i$ shortens the length of the sequence $(a, x_{i_1}, \dots, x_{i_{k}},b)$, which is equivalent to saying that $\partial_i \{x_{i_1}, \dots, x_{i_{k}}\} \in C_{\ast}(K'_{\ell}(a, b))$. 
\begin{thm}
\label{thm4.2}
For $\ell\ge 3$, the above homomorphism 
\[
t:(C_\ast(K_\ell(a, b), K'_{\ell}(a, b)),-\partial )\longrightarrow (MC_{\ast+2,\ell}(a, b),\partial )
 \]
  is a chain map. Furthermore, it is an isomorphism for $\ast\ge 0$.
\end{thm}
\begin{proof}
We can show that the homomorphism $t$ is a chain map by the same computation as in the proof of Proposition \ref{prop3.4}. Next we show that $t$ is  injective. Suppose that 
\[
t\left(\sum_{\alpha=1}^N c_\alpha [\{x_{j_1}^{\alpha}, \dots, x_{j_{k}}^{\alpha}\}] \right)=0.
\]
We can assume that 
\[
[(x_0, x_{j_1}^{\alpha}, \dots, x_{j_{k}}^{\alpha}, x_l)] = [(x_0, x_{j_1}^{\alpha'} \dots, x_{j_{k}}^{\alpha'},x_l)] \ne 0
\] 
for any $1 \leq \alpha, \alpha' \leq N$ as in the proof of Proposition \ref{prop3.4}.  By Lemma \ref{lem4.1}, it turns out that their indices coincide.
Hence we have
\[
\sum_{\alpha=1}^N c_\alpha\{x_{j_1}^{\alpha}, \dots, x_{j_{k}}^{\alpha}\}=0,
\]
which implies that $t$ is injective. To prove surjectivity, let $(a, x_{i_1}, \dots, x_{i_{k+1}}, b)$ be a base element of $MC_{k+2, l}(a, b)$. 
Since we have 
\[
\mathrm{L}(a, x_{i_1}, \dots, x_{i_{k+1}}, b)=\ell,
\] 
there exist a path 
\[
(a, x_1,\dots, x_{\ell-1},b)\in P_{\le \ell}(a, b)
\]
such that
\[
(a, x_{i_1}, \dots, x_{i_{k+1}}, b) \prec (a, x_1, \dots, x_{\ell-1}, b).
\]
Hence we obtain 
\[
\{x_{i_1}, \dots, x_{i_{k+1}}\} \in K_\ell(a, b),
\]
and we also have
\[
\{x_{i_1}, \dots, x_{i_{k+1}}\}\not\in K'_{\ell}(a, b).
\]
Therefore $t$ is an isomorphism on the module of each dimension.
\end{proof}
By Theorem \ref{thm4.2}, we can compute most part of magnitude homology for arbitrary graphs as follows.
\begin{cor}
\label{cor4.3}
We have
\[
MH_{k, \ell}(a, b)\cong H_{k-2}(K_\ell(a, b), K'_\ell(a, b))
\]
for $k,\ell\ge3$. Moreover, for $k = 2$, we also have
\[
MH_{2, \ell}(a, b) \cong \begin{cases}H_{0}(K_\ell(a, b), K'_\ell(a, b)) & {\rm if}\ d(a, b) < \ell, \\ \tilde{H}_{0}(K_\ell(a, b))& {\rm if}\ d(a, b) = \ell,\end{cases}
\]
where $\tilde{H}_{\ast}$ denotes the reduced homology group.
\end{cor}
\begin{proof}
Suppose that $\ell, k\ge 3$.
Since the homology of $C_\ast(K_\ell(a, b), K'_{\ell}(a, b))$ does not depend on the sign of the boundary map, we have
\[  MH_{k,\ell}(a, b)\cong H_{k-2}(K_\ell(a, b),K'_{\ell}(a, b)) \]
from Theorem \ref{thm4.2}. Hence the statement follows. In case $k=2$, we consider a sequence of submodules 
\[
{\rm Im} \ \partial_{3} \subset {\rm Ker} \ \partial_{2} \subset MC_{2, \ell}(a, b),
\]
where $\partial_{\ast}$ denotes differentials of $MC_{\ast, \ell}(a, b)$. Then we have a short exact sequence
\[
0 \longrightarrow {\rm Ker} \ \partial_{2}/ {\rm Im} \ \partial_{3} \longrightarrow MC_{2, \ell}(a, b)/{\rm Im}\  \partial_{3} \longrightarrow  MC_{2, \ell}(a, b)/ {\rm Ker}\ \partial_{2} \longrightarrow 0,
\]
which is isomorphic to 
\[
0 \longrightarrow MH_{2, \ell}(a, b) \longrightarrow H_{0}(K_\ell(a, b),K'_{\ell}(a, b)) \longrightarrow {\rm Im}\ \partial_{2} \longrightarrow 0.
\]
Then, by the definition of $\partial_{2}$, we have
\[
{\rm Im} \ \partial_{2} \cong \begin{cases} 0 & {\rm if}\ d(a, b) < \ell, \\ \mathbb{Z} & {\rm if}\ d(a, b) = \ell .\end{cases}
\]
Hence we obtain 
\[
H_{0}(K_\ell(a, b),K'_{\ell}(a, b))  \cong \begin{cases} MH_{2, \ell}(a, b) & {\rm if}\ d(a, b) < \ell, \\   MH_{2, \ell}(a, b) \oplus \mathbb{Z} & {\rm if}\ d(a, b) = \ell. \end{cases}
\]
Then we have
\[
MH_{2,\ell}(a, b)\cong H_0(K_\ell(a, b), K'_\ell(a, b))
\]
for $d(a, b)<\ell$. 
If $d(a, b)=\ell$, we see 
\[
H_0(K_\ell(a, b), K'_\ell(a, b))=H_0(K_\ell(a, b))
\]
since $K'_\ell(a, b)=\emptyset$. Therefore, it also holds that
\[MH_{2,\ell}(a, b)\cong \tilde{H}_0(K_\ell(a, b)).
\]
\end{proof}
Finally, we give an example of computation using the above method.
\begin{exam}
A graph $Sq_2$ (\cite[pp 14-15]{Lein2}) is described in Figure 3.
We compute the magnitude homology $MH_{\ast,4}(Sq_2)$.
\begin{center}
\begin{figure}[ht]
\includegraphics[width=9cm]{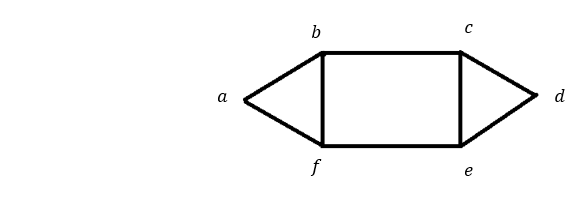}
\caption{Graph $Sq_2$.}
\end{figure}
\end{center}
The magnitude chain complex can be decomposed into 36 components which are classified into 8 types
\[(a, a),(a, b),(a, c),(a, d),(b, b),(b, c),(b, f),(b, e).\]

First, consider $(a, a)$-type components.
We have two $(a, a)$-type components, $(a, a)$- and $(d, d)$-component. 
These components have the same chain complexes by symmetry.
The paths belonging to $(a,a)$-component whose length are 4 are
\[(a, b, a, b, a),(a, b, a, f, a),(a, b, c, b, a),(a, b, f, b, a),\]
\[(a, f, a, f, a),(a, f, a, b, a),(a, f, e, f, a),(a, f, b, f, a).\]
The paths whose length are less than 4 are not related to the homology since they vanish by the quotient operation.
Assigning $(4-2)=2$-simplices for this 8 paths, 
we can construct the simplicial complex $K_4(a, a)$ and the subcomplex $K'_4(a, a)$ as shown in Figure 4.
\begin{center}
\begin{figure}[h]
\includegraphics[width=10.3cm]{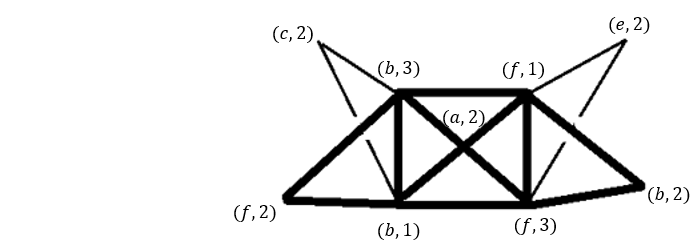}
\caption{The 2-dimensional simplicial complex $K_4(a, a)$ and the subcomplex $K'_4(a, a)$ (bold part).}
\end{figure}
\end{center}
By Corollary \ref{cor4.3}, we have
\[MH_{k, 4}(a, a)\cong \tilde{H}_{k-2}(|K_4(a, a)|/|K'_4(a, a)|)\textrm{ for }k\ge2. \]
Since $|K_4(a, a)|/|K'_4(a, a)|$ is homotopy equivalent to $S^2\vee S^2\vee S^2 \vee S^2\vee S^2\vee S^2$,
we have
\[
MH_{k, 4}(a, a)
\cong
\begin{cases}
\mathbb{Z}^6 & (k=4), \\
0 & (k\ne 4, k\ge 2).
\end{cases}
\]
In addition, we can also identify the generators of the magnitude homology from Figure 4.

$(a, d)$-type is one of types which give non-trivial elements to the 3rd magnitude homology. 
There are two $(a, d)$-type components, or $(a, d)$-component and $(d, a)$-component.
The paths belong to $(a, d)$-component are
\[(a, b, c, e, d), (a, b, f, e, d), (a, f, b, c, d), (a, f, e, c, d).\]
The resulting simplicial complexes $K_4(a,d)$ and $K'_4(a,d)$ are shown in Figure 5.
\begin{center}
\begin{figure}[h]
\includegraphics[width=11cm]{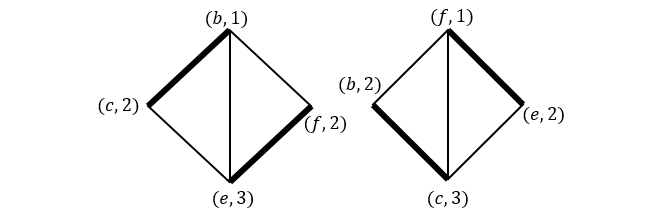}
\caption{The simplicial complex $K_4(a, d)$ and the subcomplex $K'_4(a, d)$ (bold part).}
\end{figure}
\end{center}
$|K_4(a, d)|/|K'_4(a, d)|$ is homotopy equivalent to $S^1\vee S^1$, so we have
\[
MH_{k, 4}(a, d)
\cong
\begin{cases}
\mathbb{Z}^2  & (k=3),\\
0 &( k\ne3, k\ge 2 ).
\end{cases}
\]
We can also compute the other components, and the result is following.
\begin{center}
	\begin{tabular}{|l|r|r|r|r|r|r|r|r|} \hline
		rank&(a, a)&(a, b)&(a, c)&(a, d)&(b, b)&(b, c)&(b, f)&(b, e) \\ \hline
		$k=2$&0&0&0&0&0&0&0&0 \\
		$k=3$&0&0&8&4&0&0&0&0 \\
		$k=4$&12&40&0&0&32&0&20&8 \\ \hline
	\end{tabular}
\end{center}
The rank of the 3rd magnitude homology is 12 and that of the 4th magnitude homology is 112, so that these coincide with the result \cite[TABLE 3]{Lein2}. 
\end{exam}

\section*{Acknowledgements}
The authors would like to express gratitude to Professor Osamu Saeki for bringing two of us together and supporting our research. Especially, the second author is grateful to Professor Saeki for supervising him as a master course student. The second author is  also grateful to Naoki Kitazawa, Hiroaki Kurihara and Dominik Wrazidlo for  helpful correction and comments. The first author is supported by the Program for Leading Graduate Schools, MEXT, Japan. He is also supported by JSPS KAKENHI Grant Number 17J01956.

\end{document}